\documentclass[orivec]{llncs}
\let\gid\ggid

\usepackage{amsmath,amssymb,latexsym}
\usepackage[all]{xy}
\usepackage{bussproofs}
  \EnableBpAbbreviations
\usepackage{color}
  \definecolor{darkblue}{rgb}{0.0,0.0,0.3}

\usepackage[numbers]{natbib}

\usepackage[bookmarks=true,pdfencoding=auto]{hyperref}
\hypersetup{pdfauthor={Tudor Protopopescu}, pdftitle={An Arithmetical Interpretation of Verification and Intuitionistic Knowledge}}
\usepackage[USenglish]{babel}
\usepackage{csquotes}
\usepackage{graphicx}
\usepackage{float}
\usepackage{fixltx2e}
\usepackage[inline]{enumitem}
\setlist{topsep=1ex}
\usepackage[capitalise]{cleveref} 
\spnewtheorem{customcase}[case]{Case}{\itshape}{ }
\newenvironment{ccase}[1]{\renewcommand\thecustomcase{#1}\customcase}{\endcustomcase}
\newlist{axiomlist}{enumerate}{1}
\setlist[axiomlist]{label=A\arabic*.,noitemsep,start=0,leftmargin=*,widest=A5.,ref=A\arabic*}
\crefname{axiom}{Axiom}{Axioms}
\Crefname{axiom}{Axiom}{Axioms}
\crefname{rule}{Rule}{Rules}
\Crefname{rule}{Rule}{Rules}

\makeatletter
\@addtoreset{case}{theorem}
\@addtoreset{case}{lemma}
\@addtoreset{case}{proposition}
\@addtoreset{case}{corollary}
\makeatother

\newcommand{\imp}{\rightarrow}

\newcommand{\Imp}{\Rightarrow}

\newcommand{\BImp}{\Leftrightarrow}

\newcommand{\proves}{\vdash}

\newcommand{\bx}{\Box}
\newcommand{\bk}{{\bf K}}
\newcommand{\bv}{{\bf V}}

\newcommand{\sub}{\subseteq}

\newcommand{\prf}{\ensuremath{\mathsf{Prf}}}
\newcommand{\bew}{\ensuremath{\mathsf{Bew}}}
\newcommand{\ver}{\ensuremath{\mathsf{Ver}}}

\newcommand{\siv}{\ensuremath{\mathsf{S4}}}
\newcommand{\sivv}{\ensuremath{\mathsf{S4V}}}
\newcommand{\sivvg}{\ensuremath{\mathsf{S4Vg}}}

\newcommand{\sivvm}{\ensuremath{\mathsf{S4V^-}}}

\newcommand{\sivvmg}{\ensuremath{\mathsf{S4V^-g}}}

\newcommand{\LP}{\ensuremath{\mathsf{LP}}}
\newcommand{\LPV}{\ensuremath{\mathsf{LPV}}}

\newcommand{\lpv}{\ensuremath{\mathsf{LPV}}}
\newcommand{\lpvm}{\ensuremath{\mathsf{LPV^-}}}

\newcommand{\cs}{\ensuremath{\mathcal{CS}}}
\newcommand{\IPC}{\ensuremath{\mathsf{IPC}}}
\newcommand{\iel}{\ensuremath{\mathsf{IEL}}}
\newcommand{\ielm}{\ensuremath{\mathsf{IEL^-}}}

\newcommand{\PA}{\ensuremath{\mathsf{PA}}}

\newcommand{\conpa}{\ensuremath{\mathsf{Con(PA)}}}

\newcommand{\mct}[1]{\mathcal{#1}}

\newcommand{\sft}[1]{\textsf{#1}}

\renewcommand{\gg}{\Gamma}
\newcommand{\gd}{\Delta}

\newcommand{\gid}{\gg \Imp \gd}

\newcommand{\gn}[1]{\ulcorner {#1} \urcorner}

\begin{document}

\mainmatter

\title{An Arithmetical Interpretation of Verification and Intuitionistic Knowledge}
\author{Tudor Protopopescu
  \thanks{In \emph{Logical Foundations of Computer Science}, Lecture Notes in Computer Science 9537, 317--330, Springer, 2016.
Due to several publisher's errors the publisher failed to include a corrected version of the paper for publication. 
The paper published in LNCS 9537 is an uncorrected proof which contains various errors introduced by the publisher, and lacks changes submitted to, and approved for inclusion by, the publisher.
This is the version that should have been published, and should be considered the official one.
}}
\date{}
\institute{The Graduate Center, City University of New York\\
\email{tprotopopescu@gradcenter.cuny.edu}}
\maketitle

\begin{abstract}
Intuitionistic epistemic logic introduces an epistemic operator, which reflects the intended BHK semantics of intuitionism, to intuitionistic logic.
The fundamental assumption concerning intuitionistic knowledge and belief is that it is the product of verification.
The BHK interpretation of intuitionistic logic has a precise formulation in the Logic of Proofs and its arithmetical semantics.
We show here that this interpretation can be extended to the notion of verification upon which intuitionistic knowledge is based, thereby providing the systems of intuitionistic epistemic logic extended by an epistemic operator based on verification with an arithmetical semantics too. 
\end{abstract}

\section{Introduction}\label{intro}

The intended semantics for intuitionistic logic is the Brouwer-\-Hey\-ting-\-Kol\-mo\-gorov (BHK) interpretation, which holds that a proposition is true if proved.
The systems of intuitionistic epistemic logic, the \iel\ family introduced in \cite{Artemov2014c}, extend intuitionistic logic with an epistemic operator and interpret it in a manner reflecting the BHK semantics.
The fundamental assumption concerning knowledge interpreted intuitionistically is that knowledge is the product of verification, where a verification is understood to be a justification sufficient to warrant a claim to knowledge which is not necessarily a strict proof.

In \cite{Artemov2014c} the notion of verification was treated intuitively.
Here we show that verification can also be given an arithmetical interpretation, thereby showing that the notion of verification assumed in an intuitionistic interpretation of knowledge has an exact model.

Following G\"odel \cite{Godel1933} it is well known that intuitionistic logic can be embedded into the classical modal logic \siv\ regarded as a provability logic.
Artemov \cite{Artemov2001} formulated the Logic of Proofs, \LP, and showed that \siv\ in turn can be interpreted in \LP, and that \LP\ has an arithmetical interpretation as a calculus of explicit proofs in Peano Arithmetic \PA.
\footnote{As opposed to provability in \PA, the calculus of which is the modal logic \sft{GL}, see \cite{Boolos1993}.} 
Accordingly this makes precise the BHK semantics for intuitionistic logic.
Intuitionistic logic, then, can be regarded as an implicit logic of proofs, and its extension with an epistemic/verification operator in the systems \ielm\ and \iel\ (given in \cref{sec:iel}) can be regarded as logics of implicit proofs, verification and their interaction.

This is of interest for a number of reasons.
It shows that the notion of verification on which intuitionistic epistemic logic is based is coherent and can be made concrete, and does so in a manner consonant with the intended BHK interpretation of the epistemic operator.
Further, given intuitionistic logic's importance in computer science as well as the need for a constructive theory of knowledge, finding a precise provability model for verification and intuitionistic epistemic logic (see \cref{sec:arithmetic-semantics}) is well-motivated.

\section{Intuitionistic Epistemic Logic}
\label{sec:iel}

According to the BHK semantics a proposition, $A$, is true if there is a proof of it and false if the assumption that there is a proof of $A$ yields a contradiction.
This is extended to complex propositions by the following clauses:

\begin{itemize}[noitemsep]
\item a proof of $A \land B$ consists in a proof of $A$ and a proof of $B$;
\item a proof of $A \lor B$ consists in giving either a proof of $A$ or a proof of $B$;
\item a proof of $A \imp B$ consists in a construction which given a proof of $A$ returns a proof of $B$;
\item $\neg A$ is an abbreviation for $A \imp \bot$, and $\bot$ is a proposition that has no proof.
\end{itemize}

The salient property of verification-based justification, in the context of the BHK semantics, is that it follows from intuitionistic truth, hence 
\[
A \imp \bk A\tag{Co-Reflection}
\]
is valid on a BHK reading. 
Since any proof is a verification, the intuitionistic truth of a proposition yields that the proposition is verified.

By similar reasoning the converse principle 
\[
\bk A \imp A\tag{Reflection}
\]
is not valid on a BHK reading.
A verification need not be, or yield a method for obtaining, a proof, hence does not guarantee the intuitionistic truth of a proposition. 
Reflection expresses the factivity of knowledge in a classical language, intuitionistically factivity is expressed by
\[
\phantom{Intuitionistic Factivity}\bk A \imp \neg\neg A.\tag{Intuitionistic Factivity}
\]

The basic system of intuitionistic epistemic logic, incorporating minimal assumptions about the nature of verification, is the system \ielm.
\ielm\ can be seen as the system formalising intuitionistic belief.

\begin{definition}[\ielm] \ 
The list of axioms and rules of \ielm\ consists of:

\begin{enumerate}[noitemsep,labelindent=\parindent,label=IE\arabic*.,series=ielaxioms,start=0,leftmargin=*,widest=IE5.,align=left,ref=IE\arabic*]

 \item Axioms of propositional intuitionistic logic.

 \item $\bk(A \imp B) \imp (\bk A \imp \bk B)$

 \item\label[axiom]{ax:coref} $A \imp \bk A$

\end{enumerate}
Modus Ponens. 
\end{definition}

It is consistent with \ielm\ that false propositions can be verified.
It is desirable, however, that false propositions not be verifiable; to be a logic of knowledge the logic should reflect the truth condition on knowledge, i.e.\ factivity -- that it is not possible to know falsehoods.
The system \iel\ incorporates the truth condition and hence can be viewed as an intuitionistic logic of knowledge.

\begin{definition}[\iel]
The list of axioms and rules for \iel\ are those for \ielm\ with the additional axiom:

\begin{enumerate}[resume*=ielaxioms]
\item\label[axiom]{ax:tcond} $\bk A \imp \neg \neg A$.

\end{enumerate}

\end{definition}
Given \cref{ax:coref} the idea that it is not possible to know a falsehood can be equivalently expressed by $\neg \bk \bot$.
\footnote{Or indeed, $\neg(\bk A \land \neg A)$, $\neg A \imp \neg \bk A$ or $\neg\neg(\bk \imp A)$, all are equivalent to \cref{ax:tcond} given \cref{ax:coref}, see \cite{Artemov2014c}.} 
For the following we will use this form of the truth condition in place of \cref{ax:tcond}.

Kripke models were defined for both systems, and soundness and completeness shown with respect to them, see \cite{Artemov2014c}.

\section{Embedding Intuitionistic Epistemic Logic into Classical Modal Logic of Verification}
\label{sec:iel2cl}

The well known G\"odel translation yields a faithful embedding of the intuitionistic propositional calculus, \IPC, into the classical modal logic \siv.
\footnote{The soundness of the translation was proved by G\"odel \cite{Godel1933} while the faithfulness was proved by McKinsey and Tarski \cite{McKinsey1948}.
See \cite{Chagrov1997} for a semantic, and \cite{Troelstra2000} for a syntactic proof.} 
By extending \siv\ with a verification modality $\bv$, the embedding can be extended to \ielm\ and \iel, and shown to remain faithful, see \cite{Protopopescu2015}.


\sivvm\ is the basic logic of provability and verification.

\begin{definition}[\sivvm\ Axioms]\label{df:s4vmAxioms}
The list of axioms and rules of \sivvm\ consists of:

\begin{enumerate}[noitemsep,labelindent=\parindent,label=A\arabic*.,series=axioms,start=0,leftmargin=*,widest=A5.,align=left,ref=A\arabic*]
\item\label[axiom]{ax:cl} The axioms of \siv\ for $\bx$.
\item\label[axiom]{ax:bvK} $\bv(A \imp B) \imp (\bv A \imp \bv B)$
\item\label[axiom]{ax:inter} $\bx A\imp \bv A$
\end{enumerate}

\begin{enumerate}[labelindent=\parindent,label=R\arabic*.,leftmargin=*,widest=R3,align=left]
\item Modus Ponens 
\item $\bx$-Necessitation \AX$~\fCenter\proves A$\UI$\fCenter\proves \bx A.$\DP

\end{enumerate}

\end{definition}

As with \iel\ we add the further condition that verifications should be consistent.

\begin{definition}[\sivv]
\sivv\ is \sivvm\ with the additional axiom:
\footnote{\cite{Protopopescu2015} presented a stronger version of \sivv\ with $\neg \bv \bot$ instead of $\neg\bx\bv\bot$.
The weaker axiom presented here is sufficient for the embedding; one can readily check that the G\"odel translation of $\neg\bk\bot$, $\bx\neg\bx\bv\bx\bot$, is derivable in \sivv\ as formulated here.
The weaker axiom allows for a uniform arithmetical interpretation of verification. } 

\begin{enumerate}[noitemsep,resume*=axioms]

\item\label[axiom]{ax:cssty} $\neg \bx \bv \bot$.

\end{enumerate}

\end{definition}

Kripke models for each system were outlined in \cite{Protopopescu2015} and the systems shown to be sound and complete with respect to them.


For \ielm\ and \iel\ their embedding into \sivvm\ and \sivv\ respectively, is faithful.
For an \ielm\ or \iel\ formula $F$, $tr(F)$ is the translation of $F$ according to the rule 
\begin{center}
\emph{box every sub-formula} 
\end{center}
into the language of \sivvm\ or \sivv\ respectively.

\begin{theorem}[Embedding]\label{thm:embedding}
The G\"odel translation faithfully embeds \ielm and \iel\ into \sivvm and \sivv, respectively:
\[
\ielm, \iel \proves F \BImp \sivvm, \sivv \proves tr(F).
\]

\end{theorem}

\begin{proof}
See \cite{Protopopescu2015}.
\end{proof}



\section{Logics of explicit proofs and verification}\label{sec:LPV}

G\"odel \cite{Godel1933} suggested that the modal logic \siv\ be considered as a provability calculus.
This was given a precise interpretation by Artemov, see \cite{Artemov2001,Artemov2008}, who showed that explicit proofs in Peano Arithmetic, \PA, was the model of provability which \siv\ described.
The explicit counter-part of \siv\ is the Logic of Proofs \LP\ in which each $\bx$ in \siv\ is replaced by a term denoting an explicit proof.
Since intuitionistic logic embeds into \siv\ the intended BHK semantics for \IPC\ as an implicit calculus of proofs is given an explicit formulation in \LP, and hence an arithmetical semantics.
Here we show that this arithmetical interpretation can be further extended to the Logic of Proofs augmented with a verification modality, providing \sivvm\ and \sivv, and therefore \ielm\ and \iel\ with an arithmetical semantics.
Similarly to the foundational picture regarding the relation between \IPC, \siv\ and \LP\ (see \cite{Artemov2001}) we have that 

\[\iel \ \hookrightarrow \ \sivv \ \hookrightarrow \ \LPV
\footnote{\text{Similar embeddings hold for \ielm, \sivvm, and \lpvm.}} 
\]

The basic system of explicit proofs and verifications \lpvm\ is defined thus:

\begin{definition}[Explicit Language]\label{df:lpvlanguage}

The language of \lpvm\ consists of: 
\begin{enumerate}[noitemsep]
\item The language of classical propositional logic;
\item A verification operator $\bv$;
\item \emph{Proof variables}, denoted by $x, y, x_1, x_2 \ldots$; \item \emph{Proof constants}, denoted by $a, b, c, c_1, c_2 \ldots$;
\item \emph{Operations on proof terms}, building complex proof terms from simpler ones of three types:
\begin{enumerate}[noitemsep]
\item Binary operation $\cdot$ called \emph{application}; 
\item Binary operation $+$ called \emph{plus}; 
\item Unary operation $!$ called \emph{proof checker}; 
\end{enumerate}

\item \emph{Proof terms}: any proof variable or constant is a proof term; 
if $t$ and $s$ are proof terms so are $t\cdot s$, $t+s$ and $!t$.

\item \emph{Formulas}: A propositional letter $p$ is a formula;
if $A$ and $B$ are formulas then so are $\neg A$, $A \land B$, $A \lor B$, $A \imp B$, $\bv A$, $t{:}A$.

\end{enumerate}
  
\end{definition}

Formulas of the type $t{:}A$ are read as ``\emph{$t$ is a proof $A$}''.

\begin{definition}[\lpvm]\label{df:lpvmaxioms} The list of axioms and rules of \lpvm\ consists of:

\begin{enumerate}[noitemsep,labelindent=\parindent,label=E\arabic*.,series=lpvaxioms,start=0,leftmargin=*,widest=E5.,align=left,ref=E\arabic*] 

\item\label[axiom]{exax:cl} Axioms of propositional classical logic.
\item\label[axiom]{exax:application} $t{:}(A \imp B) \imp (s{:}A \imp (t\cdot s){:}B)$
\item\label[axiom]{exax:ref} $t{:}A \imp A$
\item\label[axiom]{exax:posintsp} $t{:}A \imp !t{:}t{:}A$
\item\label[axiom]{exax:plus} $t{:}A \imp (s+t){:}A$, $t{:}A \imp (t+s){:}A$
\medskip
\item\label[axiom]{exax:bvK} $\bv(A \imp B) \imp (\bv A \imp \bv B)$
\item\label[axiom]{exax:interaction} $t{:}A \imp \bv A$

\end{enumerate}

\begin{enumerate}[labelindent=\parindent,label=R\arabic*.,leftmargin=*,widest=R3,align=left,noitemsep]

\item Modus Ponens 
\item Axiom Necessitation: \AX$\fCenter\proves A$\UI$\fCenter\proves c{:}A$\DP where $A$ is any axiom and $c$ is some proof constant. 

\end{enumerate}

\end{definition}

\begin{definition}[\lpv]\label{df:lpvaxioms} 
The system \lpv\ is \lpvm\ with the additional axiom:
\begin{enumerate}[resume*=lpvaxioms] 

\item\label[axiom]{exax:cssty} $\neg t{:} \bv \bot$ 

\end{enumerate}

\end{definition}

A \emph{constant specification}, $\mathcal{CS}$, is a set $\{c_1{:}A_1, c_2{:}A_2 \dots\}$ of formulas such that each $A_i$ is an axiom from the lists above, and each $c_i$ is a proof constant.
This set is generated by each use of the constant necessitation rule in an \lpvm\ or \lpv\ proof.
The axiom necessitation rule can be replaced with a `ready made' constant specification which is added to \lpvm\ or \lpv\ as a set of extra axioms.
For such a \cs\ let $\lpvm\text{-}\cs$ and $\lpv\text{-}\cs$ mean \lpvm\ and \lpv, respectively, minus the axiom necessitation rule plus the members of \cs\ as additional axioms.

A proof term, $t$, is called a \emph{ground term} if it contains no proof variables, but is built only from proof constants and operations on those constants. 

\lpvm\ and \lpv\ are able to internalise their own proofs, that is if \[A_1 \dots A_n, y_1{:}B_1 \dots y_n{:}B_n \ \proves F\] then for some term $p(x_1 \dots x_n, y_1 \dots y_n)$
\[x_1{:}A_1 \dots x_n{:}A_n, y_1{:}B_1 \dots y_n{:}B_n \ \allowbreak\proves \allowbreak p(x_1 \dots x_n, y_1 \dots y_n){:}F,\] see \cite{Artemov2001}.
As a consequence \lpvm\ and \lpv\ have the constructive necessitation rule: for some ground proof term $t$,

\begin{prooftree}
\AX$\fCenter\proves F$
\UI$\fCenter\proves t{:}F.$
\end{prooftree}

This yields in turn:
\begin{lemma}[$\bv$ Necessitation]\label{cor:V-nec} 
$\bv$-Necessitation \AX$~\fCenter\proves A$\UI$\fCenter\proves \bv A$\DP is derivable in \lpvm\ and \lpv.
\end{lemma}

\begin{proof}
Assume $\proves A$, then by constructive necessitation $\proves t{:} A$ for some ground proof term $t$, hence by \cref{exax:interaction} $\proves \bv A$.
\end{proof}

Note that the Deduction Theorem holds for both \lpvm\ and \lpv.

\section{Arithmetical Interpretation of \lpvm\ and \lpv}
\label{sec:arithmetic-semantics} 

We give an arithmetical interpretation of \lpvm\ and \lpv\ by specifying a translation of the formulas of \lpvm\ and \lpv\ into the language of Peano Arithmetic, \PA.
We assume that a coding of the syntax of \PA\ is given.
$n$ denotes a natural number and $\overline{n}$ the corresponding numeral. 
$\overline{\gn{F}}$ denotes the numeral of the G\"odel number of a formula $F$.
For readability we suppress the overline for numerals and corner quotes for the G\"odel number of formulas, and trust that the appropriate number or numeral, as context requires, can be recovered.
\footnote{E.g.\ by techniques found in \cite{Boolos1993} and \cite{Feferman1960}.} 

\begin{definition}[Normal Proof Predicate]\label{df:prf}
A \emph{normal proof predicate} is a provably $\Delta$ formula $\prf(x,y)$ such that for every arithmetical sentence $F$ the following holds:
\begin{enumerate}[noitemsep]
\item $\PA \proves F \BImp \text{for some } n \in \omega, \prf(n, F)$

\item A proof proves only a finite number of things; i.e.\ for every $k$ the set $T(k) = \{ l | \prf(k, l) \}$ is finite.
\footnote{I.e.\ $T(k)$ is the set of theorems proved by the proof $k$.}

\item Proofs can be joined into longer proofs; i.e.\ for any $k$ and $l$ there is an $n$ s.t.\ $T(k) \cup T(l) \sub T(n)$.
\end{enumerate}

\end{definition}

\begin{example}
An example of a numerical relation that satisfies the definition of $\prf(x,y)$ is the standard proof predicate $\mathsf{Proof}(x, y)$ the meaning of which is
\begin{center}
\emph{``x is the G\"odel number of a derivation of a formula with the G\"odel number y''.}
\end{center}

\end{example}

\begin{theorem}\label{thm:prffunctions}
For every normal proof predicate $\prf(x,y)$ there exist recursive functions $\mathsf{m}(x,y)$, $\mathsf{a}(x,y)$ and $\mathsf{c}(x)$ such that for any arithmetical formulas $F$ and $G$ and all natural numbers $k$ and $n$ the following formulas hold:

\begin{enumerate}[noitemsep]
\item $(\prf(k, F \imp G) \land \prf(n, F)) \imp \prf(\mathsf{m}(k, n), G)$

\item $\prf(k, F) \imp \prf(\mathsf{a}(k, n), F)$, \ \  $\prf(n, F) \imp \prf(\mathsf{a}(k, n), F)$

\item $\prf(k, F) \imp \prf(\mathsf{c}(k), \prf(k, F)).$
\end{enumerate}

\end{theorem}

\begin{proof}
See \cite{Artemov2001}.
\end{proof}

\begin{definition}[Verification Predicate for \lpvm]\label{df:ver}
A \emph{verification predicate} is a provably $\Sigma$ formula $\ver(x)$ satisfying the following properties, for arithmetical formulas $F$ and $G$:
\begin{enumerate}[noitemsep,series=verlpvm]
\item $\PA \proves \ver(F \imp G) \imp (\ver(F) \imp \ver(G))$
\item For each \emph{n}, $\PA \proves  \prf(n,F) \imp \ver(F)$.
\end{enumerate}

\end{definition}

These are properties which a natural notion of verification satisfies.

Let $\bew(x)$ be the standard provability predicate,
\footnote{With the standard multi-conclusion proof predicate  in which a proof $p$ is a proof of $F$ if $F$ occurs somewhere in $p$.} 
and $\conpa$ be the statement which expresses that \PA\ is consistent, i.e.\ $\neg\bew(\bot)$.
$\neg\conpa$ correspondingly is $\bew(\bot)$.

\begin{example}\label{ex:vermeaning} \ The following are examples of a verification predicate $\ver(x)$:
\begin{enumerate}[]
\item ``Provability in \PA'', i.e.\ $\ver(x) = \bew(x)$; for a formula $F$ $\ver(F)$ is $\exists x \prf(x, F)$.

\item ``Provability in \PA\ + \sft{Con}(\PA)'' i.e.\ $\ver(x) = \bew(\conpa \imp x)$; one example of a formula for which $\ver(x)$ holds in this sense is just the formula $\conpa$.
Such verification is capable of verifying propositions not provable in \PA.

\item ``Provability in \PA\ + $\neg$\sft{Con}(\PA)'' i.e.\ $\ver(x) = \bew(\neg\conpa \imp x)$; an example of a verifiable formula which is not provable in \PA, is the formula $\neg\conpa$.
Such verification is capable of verifying false propositions.

\item $\top$, i.e.\ $\ver(x) = \top$; that is for any formula $F$ $\ver(F) = \top$, hence any $F$ is verified.  
\end{enumerate}

\end{example}

\begin{lemma}
$\PA \proves F \ \Imp \ \PA \proves \ver(F)$.
\end{lemma}

\begin{proof}
Assume $\PA \proves F$, then by \cref{df:prf} there is an $n$ such that $\prf(n,F)$ is true, hence $\PA \proves \prf(n, F)$, and by \cref{df:ver} part 2 $\PA \proves \ver(F)$.
\end{proof}

We now define an interpretation of the language of \lpvm\ into the language of Peano Arithmetic. 
An arithmetical interpretation takes a formula of \lpvm\ and returns a formula of Peano Arithmetic; we show the soundness of such an interpretation, if $F$ is valid in \lpvm\ then for any arithmetical interpretation $^*$ $F^*$ is valid in \PA.
\footnote{A corresponding completeness theorem is left for future work, as is the development of a system with explicit verification terms, in addition to proof terms, realising the verification modality of \sivvm\ or \sivv.} 

\begin{definition}[Arithmetical Interpretation for \lpvm]\label{df:lparithmeticintr}
An \emph{arithmetical interpretation} for \lpvm\ has the following items:

\begin{itemize}[noitemsep]
\item A normal proof predicate, \prf, with the functions \sft{m}($x,y$), \sft{a}($x,y$) and \sft{c}($x$) as in \cref{df:prf} and \cref{thm:prffunctions};

\item A verification predicate, \ver, satisfying the conditions in \cref{df:ver};

\item An evaluation of propositional letters by sentences of \PA;

\item An evaluation of proof variables and constants by natural numbers.

\end{itemize}

An arithmetical interpretation is given inductively by the following clauses:

\begin{minipage}{0.5\textwidth}
\begin{align*}
(p)^* &= p \text{ an atomic sentence of } \PA \\
\bot^* &= \bot \\
(A \land B)^* &= A^* \land B^* \\
(A \lor B)^* &= A^* \lor B^* \\
(A \imp B)^* &= A^* \imp B^* 
\end{align*} 
\end{minipage}
\begin{minipage}{0.5\textwidth}
\begin{align*}
(t\cdot s)^* &= \mathsf{m}(t^*, s^*) \\
(t+s)^* &= \mathsf{a}(t^*, s^*) \\
(!t)^* &= \mathsf{c}(t^*) \\
(t{:}F)^* &= \prf(t^*, F^*) \\
(\bv F)^* &= \ver(F^*) 
\end{align*} 
\end{minipage}

\end{definition}

Let $X$ be a set of \lpvm\ formulas, then $X^*$ is the set of all $F^*$'s such that $F \in X$.
For a constant specification, $\mct{CS}$, a \emph{$\mct{CS}$-interpretation} is an \mbox{interpretation $^*$} such that all formulas from $\mct{CS}^*$ are true.
An \lpvm\ formula is \emph{valid} if $F^*$ is true under all interpretations $^*$. 
$F$ is \emph{provably valid} if $\PA \proves F^*$ under all interpretations $^*$. 
Similarly, $F$ is \emph{valid under constant specification $\mct{CS}$} if $F^*$ is true under all $\mct{CS}$-interpretations, and $F$ is \emph{provably valid under constant specification $\mct{CS}$} if $\PA \proves F^*$ under any $\mct{CS}$-interpretation $^*$.

\begin{theorem}[Arithmetical Soundness of \lpvm]\label{thm:arithmetic-sound-lpvm}
For any \mbox{$\mct{CS}$-interpretation $^*$} with a verification predicate as in \cref{df:ver} any \lpvm-\cs\ theorem, $F$, is provably valid under constant specification $\mct{CS}$:
\[ \lpvm\text{-}\mct{CS} \proves F \ \Imp \ \PA \proves F^*. \]

\end{theorem} 

\begin{proof}
By induction on derivations in \lpvm.
The cases of the \LP\ axioms are proved in \cite{Artemov2001}.

\begin{case}[$\bv(A \imp B) \imp (\bv A \imp \bv B)$]

\[
[\bv(A \imp B) \imp (\bv A \imp \bv B)]^* \equiv \ver(F \imp G) \imp (\ver(F) \imp \ver(G)). \]

But $\PA \proves \ver(F \imp G) \imp (\ver(F) \imp \ver(G))$ by \cref{df:ver}.
\end{case}

\begin{case}[$t{:}F \imp \bv F$]
\[ [t{:}F \imp \bv F]^* \equiv \prf(t^*, F^*) \imp \ver(F^*).\]

Likewise $\PA \proves \prf(t^*, F^*) \imp \ver(F^*)$ holds by \cref{df:ver}.
\end{case}

\end{proof}

This arithmetical interpretation can be extended to \lpv.
Everything is as above except to \cref{df:ver} we add the following item:

\begin{definition}[Verification Predicate for \lpv]\label{df:verlpv} \
\begin{enumerate}[noitemsep,resume=verlpvm]
\item for any \emph{n}, $\PA \proves \neg\prf(n, \ver(\bot)).$
\end{enumerate}

\end{definition}

1--3 of \cref{ex:vermeaning} remain examples of a verification predicate which also satisfies the above consistency property.
In each case respectively 
$\ver(\bot)$ is
\begin{enumerate}[noitemsep,series=ver]
\item $\bew(\bot)$
\item $\bew(\neg\bew(\bot) \allowbreak \imp \allowbreak \bot)$, i.e.\ $\bew(\neg\conpa)$
\item $\bew(\neg\neg\bew(\bot) \imp \bot)$, i.e.\ $\bew(\conpa)$.
\end{enumerate}
All of these are false in the standard model of \PA, and hence not provable in \PA, hence for each $n$ $\PA \proves \neg \prf(n, \ver(\bot))$.

4. $\ver(\bot) = \top$, is not an example of a verification predicate for \lpv\ in the sense of \cref{df:verlpv}: $\ver(\bot)$ would be provable in \PA, and hence there would be an $n$ for which $\PA \proves \prf(n,\ver(\bot))$ holds, which contradicts \cref{df:verlpv}.

\begin{theorem}[Arithmetical Soundness of \lpv]\label{thm:arithmetic-sound-lpv}
For any \mbox{$\mct{CS}$-interpretation $^*$} with a verification predicate as in \cref{df:verlpv}, if $F$ is an \lpv-\cs\ theorem then it is provably valid under constant specification $\mct{CS}$:
\[ \lpv\text{-}\cs \proves F \ \Imp \ \PA \proves F^*. \]

\end{theorem} 

\begin{proof}

Add to the proof of \cref{thm:arithmetic-sound-lpvm} the following case:

\begin{ccase}{3}[$\neg t{:}\bv\bot$]
\[
[\neg t{:}\bv\bot]^* \equiv \neg \prf(n, \ver(\bot)).
\]

$\PA \proves \neg \prf(n, \ver(\bot))$ holds by \cref{df:verlpv}.
\end{ccase}

\end{proof}

\section{Sequent Systems for \sivvm\ and \sivv}
\label{sec:sequent}

We give a sequent formulation of \sivvm\ and \sivv.
We will denote the sequent formulations by \sivvmg, \sivvg\ respectively.

A sequent is a figure, $\gg \Imp \gd$, in which $\gg, \gd$ are multi-sets of formulas. 

\begin{definition}[\sivvmg ]\label{df:sivvmg}

The axioms for the system \sivvmg\ are: 
\paragraph{Axioms}

\begin{center}
 \begin{tabular}{lllll}

  \AxiomC{$P \Imp P$, $P$~atomic}
  \DisplayProof
& &  \AxiomC{$ \bot \Imp$}
   \DisplayProof    
 \end{tabular}
\end{center}

The structural and propositional rules are those of the system \textbf{G1c} from \cite{Troelstra2000}. 
The modal rules are:

\paragraph{$\bx$-Rules}

\begin{center}
 \begin{tabular}{llll}

\AxiomC{$\gg, X \Imp \gd$}
\RightLabel{$(\bx \Imp)$}
  \UnaryInfC{$\gg, \bx X \Imp  \gd$}
  \DisplayProof
& \phantom{1234} & \phantom{1234}
				    & \AxiomC{$\bx \gg \Imp X$}
				      \RightLabel{$(\Imp \bx)$}
				      \UnaryInfC{$ \bx \gg \Imp \bx X $}
				      \DisplayProof

\end{tabular}
\end{center}

\medskip

\paragraph{$\bv$-Rule}

\begin{center}
 \begin{tabular}{lll}

&
				      & \AxiomC{$\bx \Theta, \gg \Imp X$}
					\RightLabel{$(\Imp \bv)$}
					\UnaryInfC{$\bx \Theta, \bv \gg \Imp \bv X$}
					\DisplayProof
\end{tabular}
\end{center}

\medskip

\paragraph{Interaction-Rule}

\begin{center}
 \begin{tabular}{l}

\AxiomC{$\gg, \bv X \Imp \gd$}
\RightLabel{$(\bv/\bx \Imp)$}
  \UnaryInfC{$\gg, \bx X \Imp  \gd$}
  \DisplayProof
\end{tabular}
\end{center}

\end{definition}

\begin{definition}[\sivvg\ Rules]\label{df:sivvg}

The system \sivvg\ is the system \sivvmg\ with the additional axiom:

\paragraph{Weak Inconsistency Elimination}

\begin{center}
 \begin{tabular}{l}

\AxiomC{$\gg \Imp \bx\bv\bot $}
\RightLabel{$( \Imp \Box{\bf V})$}
  \UnaryInfC{$\gg \Imp $}
  \DisplayProof

\end{tabular}
\end{center}

\end{definition}

Soundness can be shown by induction on the rules of \sivvmg\ and \sivvg.
Completeness and cut-elimination can be shown in a manner similar to that of \cite{Artemov2006}.
\footnote{See also \cite{Mints2000} for another example of the method.} 

\section{Realisation of \sivvm\ and \sivv}
\label{sec:realisation}

Here we show that each $\bx$ in an \sivvm\ or \sivv\ theorem can be replaced with a proof term so that the result is a theorem of \lpvm\ or \lpv, and hence that \ielm\ and \iel\ each have a proof interpretation.
The converse, that for each \lpvm\ or \lpv\ theorem if all the proof terms are replaced with $\bx$'s the result is a theorem of \sivvm\ or \sivv\ also holds.

\begin{definition}[Forgetful Projection]
The \emph{forgetful projection}, $F^0$ of an \lpvm\ or \lpv\ formula is the result of replacing each proof term in $F$ with a $\bx$.
\end{definition}

\begin{theorem}
  $\lpvm, \lpv \proves F \ \Imp \ \sivvm, \sivv \proves F^0$ respectively.
\end{theorem}

\begin{proof}
  By induction on \sivvm\ derivations. 
  The forgetful projections of \cref{exax:application,exax:plus,exax:posintsp,exax:ref,exax:interaction} are $\bx(A \imp B) \imp (\bx A \imp \bx B)$, $\bx A \imp A$, $\bx A \imp \bx \bx A$, $\bx A \imp \bx A$ and $\bx A \imp \bv A$ respectively, which are all provable in \sivvm.
The forgetful projection of $\neg t{:}\bv \bot$ is $\neg\bx\bv\bot$ which is provable in \sivv.
The rules are obvious.
\end{proof}

\begin{definition}[Realisation]
A \emph{realisation}, $F^r$, of an \sivvm\ or \sivv\ formula $F$ is the result of substituting a proof term for each $\bx$ in $F$, such that if $\sivvm,\sivv \proves F$ then $\lpvm,\lpv \proves F^r$ respectively.
\end{definition}

\begin{definition}[Polarity of Formulas]
 Occurrences of $\bx$ in $F$ in $G \imp F$, $F \land G$, $G \land F$, $F \lor G$, $G \lor F$, $\bx G$ and $\gg \Imp \gd, F$ have the same polarity as the occurrence of $\bx$ in $F$.

 Occurrences of $\bx$ in $F$ from $F \imp G$, $\neg F$ and $F, \gg \Imp \gd$ have the polarity opposite to that of the occurrence of $\bx$ in $F$.
\end{definition}

\begin{definition}[Normal Realisation]
A realisation $r$ is called \emph{normal} if all negative occurrences of $\bx$ are realised by proof variables.
\end{definition}

The informal reading of the \siv\ provability modality $\bx$ is existential, $\bx F$ means `there is a proof of $F$' (as opposed to the Kripke semantic reading which is universal, i.e.\ `$F$ holds in all accessible states'), normal realisations are the ones which capture this existential meaning, see \cite{Artemov2001}.  

The realisation theorem, \cref{thm:sivvmrealisation}, shows that if a formula $F$ is a theorem of \sivvm\ then there is a substitution of proof terms for every $\bx$ occurring in $F$ such that the result is a theorem of \lpvm.
This means that every $\bx$ in \sivvm\ can be thought of as standing for a (possibly complex) proof term in \lpvm, and hence, by \cref{thm:arithmetic-sound-lpvm}, implicitly represents a specific proof in \PA.
The proof of the realisation theorem consists in a procedure by which such a proof term can be built, see \cite{Artemov1995,Artemov2001,Kuznets2006,Fitting2005}.
Given a (cut-free) proof in \sivvmg\ we show how to assign proof terms to each of the $\bx$'s occurring in the \sivvmg\ proof so that each sequent in the proof corresponds to a formula provable in \lpvm; this is done by constructing a Hilbert-style \lpvm\ proof for the formula corresponding to each sequent, so as to yield the desired realisation.

Occurrences of $\bx$ in an \sivvmg\ derivation can be divided up into \emph{families} of related occurrences.
Occurrences of $\bx$ are related if they occur in related formulas of premises and conclusions of rules. 
A family of related occurrences is given by the transitive closure of such a relation. 
A family is called \emph{essential} if it contains at least one occurrence of $\bx$ which is introduced by the $(\Imp \bx)$ rule. 
A family is called \emph{positive (respectively negative)} if it consists of positive (respectively negative) occurrences of $\bx$. 
It is important to note that the rules of \sivvmg\ preserve the polarities of $\bx$. 
Any $\bx$ introduced by $(\Imp \bx)$ is positive, while $\bx$'s introduced by $(\bx \Imp)$ and the interaction rule are negative.


\begin{theorem}[\sivvm\ Realisation]\label{thm:sivvmrealisation}
  If $\sivvm \proves F$ then $\lpvm \proves F^r$ for some normal realisation~$r$.
\end{theorem}

\begin{proof}
 
If $\sivvmg \proves F$ then there exists a cut-free sequent proof, $\mathcal{S}$, of the sequent $\ \Imp F$.
The realisation procedure described below (following \cite{Artemov1995,Artemov2001}) describes how to construct a normal realisation $r$ for any sequent in $\mathcal{S}$.

\textbf{Step 1.} In every negative family and non-essential positive family replace each occurrence of $\bx B$ by $x{:}B$ for a fresh proof variable $x$.

\textbf{Step 2.} Pick an essential family, $f$, and enumerate all of the occurrences of the rule $(\Imp \bx)$ which introduce $\bx$'s in this family.
Let $n_f$ be the number of such introductions.
Replace all $\bx$'s of family $f$ by the proof term $v_1 + \ldots + v_{n_{f}}$ where $v_i$ does not already appear as the result of a realisation. 
Each $v_i$ is called a \emph{provisional} variable which will later be replaced with a proof term.

After this step has been completed for all families of $\bx$ there are no $\bx$'s left in $\mathcal{S}$.  

\textbf{Step 3.} This proceeds by induction on the depth of a node in $\mathcal{S}$.
For each sequent in $\mathcal{S}$ we show how to construct an \lpvm\ formula, $F^r$, corresponding to that sequent, such that $\lpvm \proves F^r$.

The realisation of a sequent $\mathcal{G} = \gid$ is an \lpvm\ formula, $\mathcal{G}^r$, of the following form: 
\[
 A^r_1 \land \ldots \land A^r_n \imp B^r_1 \lor \ldots \lor B^r_m
\]
The $A^r$'s and $B^r$'s denote realisations already performed.
Let $\gg^r, \Theta^r$ stand for conjunctions of formulas and $\gd^r$ for disjunctions of formulas; $\gg^r$ prefixed with a $\bv$ stands for conjunctions of $\bv$'ed formulas, i.e.\ $\bv \gg^r_n = \bv A^r_1 \land \ldots \land \bv A^r_n$.
Similarly $\vec{x}{:}\Theta^r_p$ stands for $x_1{:}C^r_1 \land \ldots \land x_p{:}C^r_p$.

The cases realising the rules involving the propositional connectives and $\bx$ are shown in \cite{Artemov2001}
\footnote{The procedure described in \cite{Artemov2001} gives an exponential increase in the size of the derivation of the desired $F^r$.
\cite{Kuznets2006} describes a modification of the procedure which gives only a polynomial increase.} 
(including how to replace provisional variables with terms).
Let us check the rules involving $\bv$. 

\begin{case}[Sequent $\mathcal{G}$ is the conclusion of a $(\Imp \bv)$ rule: $\bx \Theta, \bv \gg \Imp \bv X $]

\
\[
\mathcal{G}^r = (\vec{x}{:} \Theta^r_p \land \bv\gg^r_n) \imp \bv X^r.
\]

Now $\lpvm \proves ((\vec{x}{:} \Theta^r_p \land \gg^r_n) \imp X^r) \Imp \lpvm \proves ((\vec{x}{:}  \Theta^r_p \land \bv\gg^r_n) \imp \bv X^r)$, hence by the induction hypothesis the realisation of the premise of the rule, $(\vec{x}{:}  \Theta^r_p \land\gg^r_n) \imp X^r$, is provable in \lpvm, and hence:
\[\lpvm \proves (\vec{x}{:} \Theta^r_p \land \bv\gg^r_n) \imp \bv X^r.\]

\end{case}

\medskip

\begin{case}[Sequent $\mathcal{G}$ is the conclusion of a $(\bv/\bx \Imp)$ rule: $\gg, \bx X \Imp  \gd$]
\[
\mathcal{G}^r = (\gg^r_n \land x{:} X^r) \imp \gd^r_m.
\]

Since $x{:}A \imp \bv A$ is provable in \lpvm\ we have that \[
 \lpvm \proves ((\gg^r_n \land \bv X^r) \imp \gd^r_m) \Imp \lpvm \proves ((\gg^r_n \land x{:}X^r) \imp \gd^r_m).\]
By the induction hypothesis the realisation of the formula corresponding to the premise of the rule, $(\gg^r_n \land \bv X^r) \imp \gd^r_m$, is provable, and hence:
\[
 \lpvm \proves (\gg^r_n \land x{:}X^r) \imp \gd^r_m.
\]

\end{case}

\textbf{Step 4.} After applying the above three steps each $\mct{G} \in \mct{S}$ has been translated into the language of \lpvm, and been shown to be derivable in \lpvm.
Hence for the formula corresponding to the root sequent, $\Imp F$, we have that 
\[
\lpvm \proves \top \imp F^r.
\]
Since $\lpvm \proves \top$ \[
\lpvm \proves F^r.
\]

Hence if $\sivvm \proves F$ there is a normal realisation $r$ such that $\lpvm \proves F^r$.  

\end{proof}

\begin{theorem}[\sivv\ Realisation]\label{thm:sivvrealisation}
  If $\sivv \proves F$ then $\lpv \proves F^r$ for some normal realisation~$r$.
\end{theorem}

\begin{proof}

We simply add the following case to Step 3 of \cref{thm:sivvmrealisation}.
The rest is the same.

\begin{ccase}{3}[Sequent $\mathcal{G}$ is the conclusion of the Weak Inconsistency Elimination: $\gg \Imp \ \ \ $]
\[
\mathcal{G}^r = \gg^r_n \imp \bot.
\]

$\lpv \proves \gg^r_n \imp x{:}\bv\bot \Imp \lpv \proves \gg^r_n \imp \bot$, since $\lpv\proves x{:}\bv\bot \imp \bot$, hence by the induction hypothesis the realisation of the premise of the rule, $\gg^r_n \imp x{:}\bv\bot$, is provable in \lpv, and hence:
\[\lpv \proves \gg^r_n \imp \bot.\]

\end{ccase}

\end{proof}

We are finally in a position to show that the systems of intuitionistic epistemic logic, \ielm\ and \iel, do indeed have an arithmetical interpretation.

\begin{definition}\label{df:prfrealisable}
A formula of \ielm\ or \iel\ is called \emph{proof realisable} if $(tr(F))^r$ is \lpvm, respectively \lpv, valid under some normal realisation $r$. 
\end{definition}

It follows that \ielm\ and \iel\ are sound with respect to proof realisability.

\begin{theorem}\label{thm:prfrealisable}
If $\ielm, \iel\ \proves F$ then $F$ is proof realisable.
\end{theorem}

\begin{proof}
By \cref{thm:embedding} if $\ielm, \iel \proves F$ then $\sivvm, \sivv \proves tr(F)$, respectively, and by \cref{thm:sivvmrealisation} and \cref{thm:sivvrealisation} if $\sivvm, \sivv \proves tr(F)$ then $\lpvm, \lpv \proves (tr(F))^r$ respectively.
\end{proof}

By \cref{thm:arithmetic-sound-lpvm,thm:arithmetic-sound-lpv} \lpvm\ and \lpv\ are sound with respect to their arithmetical interpretation, and hence by \cref{thm:prfrealisable} so are \ielm\ and \iel.

\section{Conclusion}

Intuitionistic epistemic logic has an arithmetical interpretation, hence an interpretation in keeping with its intended BHK reading. 
Naturally verification in Peano Arithmetic, as outlined above, is not  the only interpretation of verification for which the principles of intuitionistic epistemic logic are valid. 
\ielm\ and \iel\ may be interpreted as logics of the interaction between conclusive and non-conclusive evidence, e.g.\ mathematical proof vs.\ experimental confirmation, or observation vs.\ testimony.
The question about exact interpretations for other intuitive readings of these logics is left for further investigation.

%

\bibliographystyle{abbrvnat}
\renewcommand\bibname{References}


\begin{thebibliography}{10}

\bibitem{Artemov1995}
S. Artemov.
\newblock {Operational Modal Logic}.
\newblock Technical Report MSI 95-29, Cornell University, 1995.

\bibitem{Artemov2001}
S. Artemov.
\newblock {Explicit Provability and Constructive Semantics}.
\newblock {\em Bulletin of Symbolic Logic}, 7(1):1--36, 2001.

\bibitem{Artemov2006}
S. Artemov.
\newblock {Justified Common Knowledge}.
\newblock {\em Theoretical Computer Science}, 357:4 -- 22, 2006.

\bibitem{Artemov2008}
S. Artemov.
\newblock {The Logic of Justification}.
\newblock {\em Review of Symbolic Logic}, 2008.

\bibitem{Artemov2014c}
S. Artemov and T. Protopopescu.
\newblock {Intuitionistic Epistemic Logic}.
\newblock Technical report, September 2015.

\bibitem{Boolos1993}
G. Boolos.
\newblock {\em {The Logic of Provability}}.
\newblock Cambridge University Press, 1993.

\bibitem{Kuznets2006}
V.~N. Brezhnev and R.~Kuznets.
\newblock {Making Knowledge Explicit: How Hard It Is}.
\newblock {\em Theor. Comput. Sci.}, 357(1):23--34, July 2006.

\bibitem{Chagrov1997}
A. Chagrov and M. Zakharyaschev.
\newblock {\em {Modal Logic}}.
\newblock Clarendon Press, Oxford, 1997.

\bibitem{Feferman1960}
S. Feferman.
\newblock {Arithmetization of Metamathematics in a General Setting}.
\newblock 49(1), 1960.

\bibitem{Fitting2005}
M. Fitting.
\newblock {The Logic of Proofs, Semantically}.
\newblock {\em Annals of Pure and Applied Logic}, 132:1--25, 2005.

\bibitem{Godel1933}
K.~G\"{o}del.
\newblock {An Interpretation of the Intuitionistic Propositional Calculus}.
\newblock In S~Feferman, J.~W. Dawson, W.~Goldfarb, C.~Parsons, and R.~M.
  Solovay, editors, {\em {Collected Works}}, volume~1, pages 301--303. Oxford
  University Press, 1933.

\bibitem{McKinsey1948}
J. McKinsey and A. Tarski.
\newblock {Some Theorems About the Sentential Calculi of Lewis and Heyting}.
\newblock 13(1):1--15, March 1948.

\bibitem{Mints2000}
G.~Mints.
\newblock {\em {A Short Introduction to Intuitionistic Logic}}.
\newblock Springer, 2000.

\bibitem{Protopopescu2015}
T. Protopopescu.
\newblock {Intuitionistic Epistemology and Modal Logics of Verification}.
\newblock Number 9394 in {Lecture Notes in Computer Science}, pages 295--307.
  Springer, 2015.

\bibitem{Troelstra2000}
A.S. Troelstra and H.~Schwichtenberg.
\newblock {\em {Basic {P}roof {T}heory}}.
\newblock Cambridge University Press, 2000.

\end{thebibliography}
\end{document}